\documentclass[11pt]{article}
\usepackage[utf8]{inputenc}

\usepackage{amssymb}
\usepackage{amsmath}
\usepackage{amsthm}
\usepackage{framed}
\usepackage{braket}
\usepackage{bm}
\usepackage{mathrsfs}
\usepackage{accents}
\usepackage{tocloft}
\usepackage{graphicx}
\usepackage{tikz}
\usepackage{url}
\usepackage{color}
\usepackage{xifthen}
\usepackage{xcolor}
\usepackage{framed}
\usepackage{mathtools}
\usepackage[explicit]{titlesec}
\usepackage{mdframed}
\usepackage{geometry}
\usepackage{enumerate}
\usepackage{hyperref}

\usepackage{authblk}

\usetikzlibrary{positioning}
\usetikzlibrary{calc}
\usetikzlibrary{decorations.pathreplacing}
\usetikzlibrary{decorations.fractals}
\usetikzlibrary{cd}

\newcommand{\scrN}{\mathcal{N}}

\newcommand{\I}{\mathcal{I}}

\newcommand{\scrE}{\mathcal{E}}

\newcommand{\Q}{\mathbb{Q}}
\newcommand{\R}{\mathbb{R}}

\newcommand{\range}{\operatorname{ran}}

\newcommand\forces{\Vdash}

\newcommand{\frakb}{\mathfrak{b}}
\newcommand{\frakd}{\mathfrak{d}}
\newcommand{\frakc}{\mathfrak{c}}
\newcommand{\fraks}{\mathfrak{s}}
\newcommand{\frakr}{\mathfrak{r}}
\newcommand{\frakv}{\mathfrak{v}}
\newcommand{\Pow}{\mathcal{P}}
\newcommand{\non}{\operatorname{non}}
\newcommand{\cov}{\operatorname{cov}}
\newcommand{\add}{\operatorname{add}}
\newcommand{\cof}{\operatorname{cof}}

\newcommand{\nul}{\mathcal{N}}
\newcommand{\meager}{\mathcal{M}}

\newcommand{\diam}{\mathrm{diam}}

\newcommand{\Haus}{\mathcal{H}}

\newcommand{\restrict}{\upharpoonright}

\newcommand{\ZFC}{\mathsf{ZFC}}
\newcommand{\CH}{\mathsf{CH}}
\newcommand{\Con}{\mathsf{Con}}

\newcommand{\Nho}{\mathcal{N}^h_\Omega}
\newcommand{\StdSets}{\operatorname{StdSets}}

\newcommand{\seq}[1]{{\langle#1\rangle}}
\DeclarePairedDelimiter\abs{\lvert}{\rvert}

\renewcommand\emptyset{\varnothing}
\renewcommand\subset{\subseteq}
\renewcommand{\setminus}{\smallsetminus}

\renewcommand{\le}{\leqslant}
\renewcommand{\ge}{\geqslant}

\newcommand{\needtocheck}[1][]{%
	\ifthenelse{\equal{#1}{}}{%
		\textcolor{blue}{[NeedToCheck]}%
	}{%
		\textcolor{blue}{[NeedToCheck: #1]}%
	}%
}

\newcommand{\todo}[1][]{%
	\ifthenelse{\equal{#1}{}}{%
		\textcolor{red}{[TODO]}%
	}{%
		\textcolor{red}{[TODO: #1]}%
	}%
}

\theoremstyle{definition}
\newtheorem{thm}{Theorem}
\newtheorem*{thm*}{Theorem}
\newtheorem{defi}[thm]{Definition}
\newtheorem*{defi*}{Definition}
\newtheorem{lem}[thm]{Lemma}
\newtheorem{prop}[thm]{Proposition}
\newtheorem*{lem*}{Lemma}

\newtheorem*{fact*}{Fact}
\newtheorem{rmk}[thm]{Remark}
\newtheorem*{rmk*}{Remark}
\newtheorem{cor}[thm]{Corollary}
\newtheorem*{cor*}{Corollary}
\newtheorem*{convention*}{Convention}
\newtheorem*{notation*}{Notation}
\newtheorem{question}[thm]{Question}

\definecolor{reasontext}{rgb}{0,0,0}
\definecolor{reasonbg}{rgb}{0.9,0.9,0.9}

\usepackage[backend=biber,style=numeric,sorting=nty,doi=false,isbn=false,url=false,eprint=false,maxcitenames=8,mincitenames=8,maxbibnames=1000]{biblatex}
\renewbibmacro{in:}{}

\definecolor{myred}{RGB}{230, 150, 100}
\definecolor{mygreen}{RGB}{150, 230, 100}
\definecolor{myblue}{RGB}{130, 220, 220}
\hypersetup{linkbordercolor=myred,citebordercolor=mygreen,urlbordercolor=myblue}

\title{The Hausdorff measure due to Davies and Rogers and its cardinal invariants}
\author{Tatsuya Goto}
\date{\today}

\affil{
	Institute of Discrete Mathematics and Geometry, TU Wien, \\ 
	Wiedner Hauptstrasse 8-10/104, 1040 Wien \\
	E-mail: goto.tatsuya@icloud.com \\
	ORCID: 0009-0006-4732-9614}
\begin{document}
	\maketitle
	
	\begin{abstract}
		Davies and Rogers constructed a Hausdorff measure satisfying the following property: every Borel subset of the space has measure either $\infty$ or $0$.
		In this paper, we examine cardinal invariants of their measure.
		Our main results are that the covering number of this Hausdorff measure zero ideal is below the uniformity of the ideal of closed Lebesgue measure zero ideal and its dual. Also we showed the strict inequality between them is consistent.
	\end{abstract}
	
	\textit{Keywords: } Hausdorff measures, Cardinal characteritics of the continuum
	
	\section{Introduction}
	
	Hausdorff measures, which generalize the concept of the Lebesgue measure, are important tool that can be useful in many fields.
	As for these, it is natural from a measure theoretic point of view to examine the claim that for any Hausdorff measure and any Borel set having infinite measure, there is a Borel subset of measure positive and finite inside it.
	In 1969, Davies and Rogers constructed a Hausdorff measure that is a counterexample to this.
	In this paper, we examine cardinal invariants of their measure.
	
	The following is the standard definiton of cardinal invariants.
	
	\begin{defi}
		\begin{enumerate}
			\item $(\forall^\infty n)$ and $(\exists^\infty n)$ are shorthands for $(\exists m)(\forall n \ge m)$ and $(\forall m)(\exists n \ge m)$, respectively.
			\item Define the almost domination order $\le^*$ on $\omega^\omega$ by $x \le^* y$ iff $(\forall^\infty n) x(n) \le y(n)$.
			\item We say $F \subset \omega^\omega$ is an unbounded family if $\neg(\exists g \in \omega^\omega)(\forall f \in F)(f \le^* g)$.  Put $\frakb = \min \{\abs{F} : F \subset \omega^\omega \text{ unbounded family} \}$.
			\item We say $F \subset \omega^\omega$ is a dominating family if $(\forall g \in \omega^\omega)(\exists f \in F)(g \le^* f)$. Put $\frakd = \min \{\abs{F} : F \subset \omega^\omega \text{ dominating family} \}$.
			\item $\nul$ and $\meager$ denote the Lebesgue measure zero ideal and Baire first category ideal on $2^\omega$, respectively.
			\item $\mathcal{E}$ denotes the $\sigma$-ideal on $2^\omega$ $
			\sigma$-generated by closed Lebesgue null sets.
			\item For an ideal $\I$ on a set $X$: 
			\begin{enumerate}
				\item $\add(\I)$ (the additivity of $\I$) is the smallest cardinality of a family $F$ of sets in $\I$ such that the union of $F$ is not in $\I$.
				\item $\cov(\I)$ (the covering number of $\I$) is the smallest cardinality of a family $F$ of sets in $\I$ such that the union of $F$ is equal to $X$.
				\item $\non(\I)$ (the uniformity of $\I$) is the smallest cardinality of a subset $A$ of $X$ such that $A$ does not belong to $\I$.
				\item $\cof(\I)$ (the cofinality of $\I$) is the smallest cardinality of a family $F$ of sets in $\I$ that satisfies the following condition: for every $A \in \I$, there is $B \in F$ such that $A \subset B$.
			\end{enumerate}
		\end{enumerate}
	\end{defi}
	
	The cardinal invariants defined below are so-called localization cardinals and anti-localization cardinals.
	
	\begin{defi}
		\begin{enumerate}
			\item For a set $A$ and a cardinal $\kappa$, let $[A]^{\le \kappa} = \{ B \subset A : |B| \le \kappa \}$.
			\item For $c \in (\omega+1)^\omega$ and $h \in \omega^\omega$, define $\prod c = \prod_{n \in \omega} c(n)$ and $S(c, h) = \prod_{n \in \omega} [c(n)]^{\le h(n)}$.
			\item For $x \in \prod c$ and $\varphi \in S(c, h)$, define $x \in^* \varphi$ iff $(\forall^\infty n)(x(n) \in \varphi(n))$ and define $x \in^\infty \varphi$ iff $(\exists^\infty n)(x(n) \in \varphi(n))$.
			\item $\frakc^\forall_{c, h} = \min \{ \abs{S} : S \subseteq S(c, h), (\forall x \in \prod c)(\exists \varphi \in S)(\forall^\infty n) (x(n) \in \varphi(n))  \}$.
			\item $\frakc^\exists_{c, h} = \min \{ \abs{S} : S \subseteq S(c, h), (\forall x \in \prod c)(\exists \varphi \in S)(\exists^\infty n) (x(n) \in \varphi(n))  \}$.
			\item $\frakv^\forall_{c, h} = \min \{ \abs{X} : X \subseteq \prod c, (\forall \varphi \in S(c, h))(\exists x \in X)(\exists^\infty n) (x(n) \not \in \varphi(n)) \}$.
			\item $\frakv^\exists_{c, h} = \min \{ \abs{X} : X \subseteq \prod c, (\forall \varphi \in S(c, h))(\exists x \in X)(\forall^\infty n) (x(n) \not \in \varphi(n)) \}$.
		\end{enumerate}
	\end{defi}
	
	We use the notions of Tukey relations, Polish relational systems and goodness to analyze cardinal invariants. As for these, the reader may consult \cite{6168e47a-d410-3c15-b326-a2970fc8c433}.
	
	\section{Review of Hausdorff measures}
	
	\begin{defi}
		We say $f \colon [0, \infty) \to [0, \infty)$ is a \textit{gauge function} if it is nondecreasing and right-continuous and $f(0) = 0$ holds.
	\end{defi}
	
	\begin{defi}
		Let $(X, d)$ be a metric space, $f$ be a gauge function , $\delta > 0$ be a real number.
		For $A \subset X$ and a sequence $\seq{C_n : n \in \omega}$ of subsets of $X$, we say this sequence is $\delta$-cover of $A$ if
		\[A \subset \bigcup_n C_n\text{ and }\diam(C_n) \le \delta \text{ for every $n$}.\]
		
		For a subset $A \subset X$, let
		\[
		\Haus^f_{X,\delta}(A) = \inf \left\{ \sum_{n \in \omega} f(\diam(C_n)) : \text{$\seq{C_n : n \in \omega}$ is a $\delta$-cover of $A$} \right\},
		\]
		which we call \textit{$\delta$-approximation of $f$-Hausdorff measure} of $A$. Also let
		\[
		\Haus^f_X(A) = \sup_{\delta > 0} \Haus^f_{X,\delta}(A),
		\]
		which we call \textit{$f$-Hausdorff measure of $A$}.
		
		We put $\scrN^f_X = \{ A \subset X : \Haus^f_X(A) = 0 \}$.
		
		If the context is clear, we omit the subscript $X$.
	\end{defi}

	\begin{rmk}
		Both $\delta$-approximation of $f$-Hausdorff measure and $f$-Hausdorff measure are outer measures.
	\end{rmk}

	\begin{rmk}\label{rmk:deltaapprox}
		For every metric space $(X, d)$, every gauge function $f$ and every real number $\delta > 0$, $\Haus^f(A) = 0$ holds iff $\Haus^f_\delta(A) = 0$ holds.
	\end{rmk}
	
	\section{General results on additivity and cofinality}\label{sec:additivity}
	
		Let $(X,d)$ be a separable metric space and let $f$ be a gauge function.
		We define the condition \emph{Assumption ($*$)} as follows: 
		\begin{align*}
			\text{Assumption ($*$)} \iff &(\exists \alpha>1)(\exists \mathcal{C} \subset \Pow(X) \text{ countable}) \\
			&(\forall A\subset X \text{ of finite diameter})(\forall \varepsilon>0)(\exists C\in\mathcal{C}) \\
			&(A \subset C \quad\text{and}\quad f(\diam(C))\le \varepsilon+\alpha\cdot f(\diam(A))).
		\end{align*}
		\begin{lem}
			Let $(X,d)$ be a separable metric space.
			Then there exists a countable family $\mathcal{C}$ of subsets of $X$ such that for every $A\subset X$ of finite diameter and every $\varepsilon>0$, there is some $C\in\mathcal{C}$ with
			\[
			A\subset C \text{ and } \diam(C)\le \varepsilon+2\diam(A).
			\]
			Moreover, if $X$ is also an ultrametric space, then we can strengthen this to
			\[
			A\subset C \text{ and } \diam(C)\le \varepsilon+\diam(A).
			\]
		\end{lem}
		\begin{proof}
		By separability, we can take a countable dense subset $D=\{a_n:n\in\omega\}$ of $X$.
		Let
		\[
		\mathcal{C}=\{B(a_n,q): n\in\omega,\ q\in\Q,\ q>0\},
		\]
		where $B(x, r)$ is the closed ball of center $x$ and radius $r$.
		We show that this $\mathcal{C}$ is the desired family.
		
		Let $A\subset X$ have finite diameter and let $\varepsilon>0$ be given.
		The case $A=\emptyset$ is clear.
		Assume $A\neq\emptyset$ and take $b\in A$.
		By density of $D$, choose $n$ such that $d(a_n,b)<\varepsilon/4$.
		Choose a positive rational $q$ such that
		\[
		\diam(A)+\varepsilon/4\le q<\diam(A)+\varepsilon/2.
		\]
		We claim that $C=B(a_n,q)$ is a desired member.
		
		First, we show $A\subset C$.
		Let $a\in A$. Then
		\[
		d(a_n,a)\le d(a_n,b)+d(b,a)<\varepsilon/4+\diam(A)\le q,
		\]
		so $a\in C$.
		
		Since $C$ is an open ball of radius $q$, the triangle inequality gives
		\[
		\diam(C)\le 2q\le \varepsilon+2\diam(A).
		\]
		
		If $X$ is an ultrametric space, then
		\[
		\diam(C)\le q\le \varepsilon+\diam(A).
		\]
		\end{proof}
	
		\begin{cor}
			Let $(X,d)$ be a separable metric space and $f$ be a gauge function.
			Suppose that the one of the following holds:
			\begin{enumerate}
				\item $f$ is a continuous doubling gauge function; that is, $f$ is a continuous function and there exists $r>0$ such that
				\[
				f(2x)<r f(x)\quad\text{for all }x\in[0,\infty).
				\]
				\item $(X,d)$ is an ultrametric space.
			\end{enumerate}
			Then Assumption ($*$) holds.
		\end{cor}
		\begin{proof}
		(1) Take the countable family $\mathcal{C}$ from the lemma, and let $r>0$ witness the doubling condition.
		Let $A\subset X$ have finite diameter and let $\varepsilon>0$.
		Since $f$ is right-continuous, choose $\delta>0$ such that
		\[
		f(\diam(A)+\delta)-f(\diam(A))<\varepsilon/r.
		\]
		By the choice of $\mathcal{C}$, there exists $C\in\mathcal{C}$ with $A\subset C$ and
		\[
		\diam(C)\le 2\delta+2\diam(A).
		\]
		Then
		\[
		f(\diam(C))\le f(2\delta+2\diam(A))\le r f(\delta+\diam(A))\le \varepsilon+r f(\diam(A)).
		\]
		Hence Assumption ($*$) holds.
		
		(2) Take the countable family $\mathcal{C}$ given by the lemma for ultrametric spaces.
		Let $A\subset X$ have finite diameter and let $\varepsilon>0$.
		Since $f$ is right-continuous, choose $\delta>0$ such that
		\[
		f(\diam(A)+\delta)-f(\diam(A))<\varepsilon.
		\]
		By the choice of $\mathcal{C}$, there exists $C\in\mathcal{C}$ with $A\subset C$ and
		\[
		\diam(C)\le \delta+\diam(A).
		\]
		Then
		\[
		f(\diam(C))\le f(\delta+\diam(A))\le \varepsilon+f(\diam(A)).
		\]
		Hence Assumption ($*$) holds.
		\end{proof}
		
		\begin{lem}
			Let $(X,d)$ be a separable metric space and let $f$ be a gauge function satisfying Assumption ($*$).
			Then for $A\subset X$, the following are equivalent:
			$A$ has $f$-Hausdorff measure $0$, and for every $\varepsilon>0$ there exists a sequence
			$\seq{C_n:n\in\omega}\subset\mathcal{C}$ such that
			\[
			A\subset\bigcup_{n\in\omega} C_n
			\quad\text{and}\quad
			\sum_{n\in\omega} f(\diam(C_n))\le\varepsilon.
			\]
		\end{lem}
		\begin{proof}
		The “if” direction is clear. We prove the “only if” direction.
		
		Assume $A$ has $f$-Hausdorff measure $0$, and let $\varepsilon>0$.
		By definition, we can find $\seq{A_n:n\in\omega}\subset \Pow(X)$ such that
		\[
		A\subset\bigcup_{n\in\omega} A_n
		\quad\text{and}\quad
		\sum_n f(\diam(A_n))<\varepsilon/\alpha.
		\]
		Take a sequence of positive reals $\seq{\eta_n:n\in\omega}$ such that
		\[
		\sum_n\bigl(\eta_n+\alpha\cdot f(\diam(A_n))\bigr)\le \varepsilon.
		\]
		For each $n$, choose $C_n\in\mathcal{C}$ such that $A_n\subset C_n$ and
		\[
		f(\diam(C_n))\le \eta_n+\alpha\cdot f(\diam(A_n)).
		\]
		Then we obtain $\sum_n f(\diam(C_n))\le \varepsilon$.
	\end{proof}
	
		\begin{prop}
			Let $(X,d)$ be a separable metric space and let $f$ be a gauge function. Assume Assumption ($*$) holds for these.
			Then $\mathcal{N}^f_X \preceq_\mathrm{T} \nul$.
			In particular, under this assumption,
			\[
			\add(\nul)\le \add(\mathcal{N}^f_X)
			\quad\text{and}\quad
			\cof(\mathcal{N}^f_X)\le \cof(\nul).
			\]
		\end{prop}
		\begin{proof}
			This is based on Theorem 534B of \cite{fremlin2008measure}.
			
		Let $\mathcal{S}=\prod_{n\in\omega}[\omega]^{\le 2^n}$.
		By Bartoszy\'nski's characterization \cite{Bartoszynski2010} of the null ideal,
		\[
		(\nul,\nul,\subset)\equiv_\mathrm{T}(\omega^\omega,\mathcal{S},\in^*),
		\]
		so it suffices to show $\mathcal{N}^f_X \preceq_\mathrm{T} (\omega^\omega,\mathcal{S},\in^*)$.
		Take the countable family $\mathcal{C}$ from the assumption.
		
		For each $n\in\omega$, let
		\[
		\mathcal{I}_n=\left\{ I\subset \mathcal{C}\ \text{finite}:\ \sum_{C\in I} f(\diam(C))\le 4^{-n}\right\}.
		\]
		Let $\seq{I_{n,j}:j\in\omega}$ be an enumeration of $\mathcal{I}_n$.
		
		For $S\in\mathcal{S}$, define
		\[
		\psi(S)=\bigcap_{m\in\omega}\ \bigcup_{n\ge m}\ \bigcup_{i\in S(n)}\ \bigcup I_{n,i}.
		\]
		
		We show that $\psi(S)\in \scrN^f_X$.
		\begin{align*}
			\Haus^f_\infty(\psi(S))
			&\le \sum_{n\ge m}\ \sum_{i\in S(n)} 4^{-n}
			= \sum_{n\ge m} 4^{-n}\cdot 2^n
			= 2^{-m+1}\to 0\ \ (m\to\infty),
		\end{align*}
		so this holds. Thus $\psi:\mathcal{S}\to \scrN^f_X$ is defined.
		Fix $A\in\mathcal{N}^f_X$.
		For each $n\in\omega$, we can take a sequence $\seq{C_{n,i}:i\in\omega}\subset\mathcal{C}$ such that
		\[
		A\subset \bigcup_i C_{n,i}
		\text{ and }
		\sum_i f(\diam(C_{n,i}))\le 2^{-n-1}.
		\]
		Rearranging the double sequence $\seq{C_{n,i}:n,i\in\omega}$ into a single sequence, we obtain
		a sequence $\seq{C_i:i\in\omega}$ such that
		\[
		\sum_i f(\diam(C_i))\le 1
		\text{ and }
		A\subset \bigcap_m \bigcup_{i\ge m} C_i.
		\]
		Choose a strictly increasing sequence of natural numbers $\seq{k(n):n\in\omega}$ satisfying
		\[
		\sum_{i=k(n)}^\infty f(\diam(C_i))\le 4^{-n}
		\text{ for every }n.
		\]
		
		Define $\varphi(A)\in\omega^\omega$ by
		\[
		\{C_i: k(n)\le i<k(n+1)\}=I_{n,\varphi(A)(n)}
		\text{ for each }n.
		\]
		Then $\varphi:\scrN^f_X\to \omega^\omega$ is defined and satisfies
		\begin{equation}\label{eq:asubset}
		A\subset \bigcap_m \bigcup_{n\ge m} \bigcup I_{n,\varphi(A)(n)}.
		\end{equation}
		We claim for any $A\in\mathcal{N}^f_X$ and $S\in\mathcal{S}$, the condition $\varphi(A)\in^* S$ implies
		$A\subset \psi(S)$.
		To show it, fix $x \in A$.
		By \eqref{eq:asubset}, we have $(\exists^\infty n)(\exists C \in I_{n,\varphi(A)(n)})\ x \in C$.
		On the other hand, by $\varphi(A) \in^* S$, we have $(\forall^\infty n)\ \varphi(A)(n) \in S(n)$. Thus, we have  $(\exists^\infty n)(\exists i \in S(n))(\exists C \in I_{n,i})\ x \in C$. Therefore, we have $x \in \psi(S)$.
		\end{proof}

	\section{Review of the construction due to Davies and Rogers}
	
	In this section, we review the construction due to Davies and Rogers \cite{daviesrogers}.
		
	\begin{lem}\label{lem:graph}
		Let $n > 0$ be a natural number.
		Then there is a finite graph $G$ that satisfies the following properties:
		\begin{enumerate}
			\item $G$ cannot be partitioned into $n$ independent subsets.
			\item For every function $w \colon G \to \R_{\ge 0}$, there is an independent subset $H \subset G$ such that
			$
			\sum_{g \in H} w(g)  \ge \frac14 \sum_{g \in G} w(g).
			$
			\item There is a family $\seq{H(x) : x \in G}$ of independent subsets such that 
			\begin{enumerate}
				\item $x \in H(x)$ for every $x \in G$, and
				\item $\abs{\{ x \in G : a \in H(x) \}} \ge \frac14 \abs{G}$ for every $a \in G$.
			\end{enumerate}
		\end{enumerate}
		Here, we identify $G$ and the set of vertices of $G$. 
	\end{lem}
	\begin{proof}
		The proof is the same as Lemma 1 of \cite{daviesrogers} and Example 4.3.43 of \cite{zapletal2008forcing}.
		
		Let $S^n=\{x\in\R^{n+1}:\|x\|=1\}$ be the $n$-sphere, and let $\sigma^n$ be the uniform probability measure on $S^n$.
		Set $\varepsilon_n = 1/(2\sqrt{n})$.
		
		For $x\in S^n$, define
		\[
		C(x)=\{y\in S^n : x\cdot y \ge \varepsilon_n\},
		\]
		and call it the \emph{cap} centered at $x$.
		
		For all sufficiently large $n$ and all $x\in S^n$, we have $\sigma^n(C(x))> \frac14$
		(see \cite{daviesrogers} for this computation).
		Refix such an $n$ larger than given integer.
		
		Let $m$ be a sufficiently large integer.
		Consider choosing $m$ points on $S^n$ independently at random with respect to $\sigma^n$.
		Fix a realization of these $m$ points, and let $G$ be the vertex set of a graph.
		We put an edge between $x,y\in G$ if and only if
		\[
		\|x-y\|\ge 2-\varepsilon_n^2.
		\]
		
		Then for any $x\in S^n$, the set $C(x)\cap G$ is an independent set.
		
		We prove (1) of the lemma.
		Assume that $G$ is partitioned into $n$ subsets $G_1,\dots,G_n$.
		For $i=1,\dots,n$, define
		\[
		H_i=\{h\in S^n : \text{there exists } g\in G_i \text{ with } \|h-g\|\le \varepsilon_n\}.
		\]
		Each $H_i$ is a closed subset of $S^n$.
		Since $m$ is sufficiently large and the points are random, the sets $H_i$ cover $S^n$.
		Now apply the Lusternik--Schnirelmann--Borsuk theorem
		(if $S^n$ is covered by at most $n+1$ closed sets, then one of them contains an antipodal pair).
		Hence there exist $i\in\{1,\dots,n\}$ and $h\in S^n$ such that $\pm h\in H_i$.
		Thus we can find $g,g'\in G_i$ with $\|h-g\|\le \varepsilon_n$ and $\|h+g'\|\le \varepsilon_n$.
		It follows that $g$ and $g'$ are adjacent, and (1) is proved.
		
		Next we prove (2).
		We may assume that $w(g)>0$ for some $g\in G$, and also normalize so that
		$\sum_{g\in G} w(g)=1$.
		It suffices to find an independent set $H\subset G$ such that
		$\sum_{g\in H} w(g)\ge \frac14$.
		Since $C(x)\cap G$ is independent, it is enough to show that there exists $x\in S^n$ with
		\[
		\sum_{g\in C(x)\cap G} w(g)\ge \frac14.
		\]
		
		Let $c=\sigma^n(C(x))$ (this does not depend on $x$). Then
		\begin{align*}
			\int_{S^n} \sum_{g\in C(x)\cap G} w(g)\, d\sigma^{n}(x)
			&= \sum_{g\in G} \int_{g\cdot x \ge \varepsilon_n} w(g)\, d\sigma^{n}(x) = c > \frac14.
		\end{align*}
		Therefore, for some $x\in S^n$ we have $\sum_{g\in C(x)\cap G} w(g)>\frac14$, proving (2).
		
		For (3), we claim that $H(x)=C(x)\cap G$ gives the desired condition.
		In order to prove it, we show:
		\begin{align}
		(\forall^\infty m) \left[ (\sigma^n \otimes \dots \otimes \sigma^n)\left(\left\{ x \in (S^n)^m : \frac{|\range(x) \cap C(x(i))|}m > \frac14 \text{ for every } i < m \right\}\right) > \frac12 \right]. \label{eq:half}
		\end{align}
		Once this formula is established, then we could take $m$ and the vertices in $G$ satisfying this.
		
		Fix $i < m$. By Fubini's theorem, Law of Large Numbers and $\sigma^n(C(a)) > 1/4$ for every $a \in S^n$ give		
		\[
		(\forall^\infty m) \left[ (\sigma^n \otimes \dots \otimes \sigma^n)\left(\left\{ x \in (S^n)^m : \frac{|\range(x) \cap C(x(i))|}m > \frac14 \right\}\right) > 1 - \frac{1}{3^i} \right].
		\]
		This gives \eqref{eq:half}.
	\end{proof}

	For each natural number $n$, fix one of the finite graphs $G$ obtained by substituting it into the Lemma \ref{lem:graph}  and let $G(n)$ be this graph.
	
	We define a sequence $\seq{M_n, G_n, N_n : n \in \omega}$ as follows:
	\begin{align*}
		M_0 &= 1 \\
		G_n  &= G(M_n) \\
		N_n &= \abs{G_n} \\
		M_{n+1} &= 2 N_n \text{ (for $n \in \omega$)}
	\end{align*}

	Note that
	\[
	\lim_{n \to \infty} \frac{N_0 \dots N_{n-1}}{M_0 \dots M_n} = 0 \text{ and } \lim_{n \to \infty} \frac{N_0 \dots N_n}{M_0 \dots M_n} = \infty.
	\]
	The first one follows from $M_{n+1} = 2 N_n$ and the second one follows from $N_n \ge \frac{4}{3} M_n$, which is a consequence of (1) and (2) of Lemma \ref{lem:graph}.

	Let $\Omega = \prod_{n \in \omega} G_n$.
	
	Induce a metric $\rho$ into $\Omega$ as follows: For $x, y \in \Omega, x \ne y$, letting $n$ be the first index that $x$ and $y$ differ,
	\[
	\rho(x, y) = \begin{cases}
		2^{-n+1} & \text{if $x(n)$ and $y(n)$ are adjacent in the graph $G_n$,} \\
		2^{-n} & \text{otherwise}.
	\end{cases}
	\]
	Define a gauge function $h$ as follows:
	\[
	h(0) = 0, \quad h(2^{-n}) = \frac{1}{M_0 \dots M_n} \text{ (for $n \in \omega$)}.
	\]
	For points other than these, define the value of $h$ by	 interpolating linearly.
	
	Davies and Rogers proved the following two propositions.
	
	\begin{prop}
		$\Haus_\Omega^h(\Omega) = \infty$.
	\end{prop}
	
	\begin{prop}
		For every subset $A \subset \Omega$, we have either $\Haus_\Omega^h(A) = 0$ or $\Haus_\Omega^h(A) = \infty$.
	\end{prop}

	They used the following notion to prove the above two propositions.

	Let $S \subset \Omega$ be a subset and assume that $S$ is of the form:
	\[
	S = \prod_{i \in \omega} H_i,
	\]
	and for some $j \in \omega$ we have
	\begin{itemize}
		\item For every $i < j$, $H_i$ is a singleton,
		\item $H_j \ne \emptyset$, and
		\item For every $i > j$, we have $H_i  =G_i$.
	\end{itemize}

	If $H_j$ is an independent subset of $G_j$, then we say $S$ is a \textit{standard set of rank $j$}.
	
	Note that standard sets of rank $j$ have diameter $2^{-j}$.
	
	\begin{lem}\label{lem:standardset}
		For every subset $S$ of  $\Omega$, if $\abs{S} \ge 2$, then there is a standard set containing $S$ and having same diameter.
	\end{lem}
	\begin{proof}
		See \cite{daviesrogers}.
	\end{proof}
	
	By the above lemma, the Hausdorff measure $\Haus^h_\Omega$ satisfies Assumption (*) defined in Section \ref{sec:additivity}. Thus we have $\add(\nul) \le \add(\Nho)$ and $\cof(\Nho) \le \cof(\nul)$.
	
	\section{$\ZFC$ results}
	
	\begin{thm}
		$\cov(\Nho) \le \non(\mathcal{E})$ and $\cov(\mathcal{E}) \le \non(\Nho)$ hold.
	\end{thm}
	\begin{proof}
		For every $n \in \omega$, let $\mu_n$ be the uniform measure on $G_n$. 
		Let $\mu$ be the product measure on $\Omega$ of $\seq{\mu_n : n \in \omega}$.
		
		For every $n \in \omega$ and $x \in G_n$, let $H(x)$ be the set defined in (3) of Lemma \ref{lem:graph}.
		
		Define a set $B$ as follows:
		\[
		B = \{ (x, y) \in \Omega^2 : (\exists^\infty n)\ y(n) \in H(x(n)) \}.
		\]
		It is sufficient to prove that
		\begin{enumerate}
			\item \label{item:vertical} For every $x \in \Omega$, $\Haus^h_\Omega(B_x) = 0$ holds.
			\item \label{item:horizontal}For every $y \in \Omega$, $\mu(\Omega \setminus B^y) = 0$ holds.
		\end{enumerate}
		Indeed, this gives a Tukey morphism $\phi \colon \Omega \to \nul_\mu$ and $\psi \colon \Omega \to \scrN^h_\Omega$ defined by $\phi(y) = \Omega \setminus B^y$ and $\psi(x) = B_x$.
		
		Since each $\Omega \setminus B^y$ is a $F_\sigma$ set, (2) implies $\Omega \setminus B^y \in \scrE_\mu$.
		
		First we observe that for every $n_0 \in \omega$,
		\begin{align*}
			\Haus^h_{\Omega,1}(B_x) &\le \Haus^h_{\Omega,1}(\{  y\in \Omega : \exists n \ge n_0\ y(n) \in H(x(n)) \}) \\
			&\le \sum_{n \ge n_0} \Haus^h_{\Omega,1}(\{ y \in \Omega : y(n) \in H(x(n)) \}) \\
			&\le \sum_{n \ge n_0} \frac{N_0 \dots N_{n-1}}{M_0 \dots M_n} \le \sum_{n \ge n_0} \frac{1}{M_0} \left(\frac12\right)^{n} \to 0 \text{ ($n_0 \to \infty$)}.
		\end{align*}
		This is because there are $N_0 \dots N_{n-1}$ many standard sets of level $n$ whose $n$-th coordinate is $H(x(n))$.
		So we have (\ref{item:vertical}).
		
		Second, for $y \in \Omega$, we consider the set
		\[
		\Omega \setminus B^y = \{ x \in \Omega : \forall^\infty n\ y(n) \not \in H(x(n)) \}.
		\]
		By independence, we have
		\[
		\mu(\Omega \setminus B^y) = \prod_{n \in \omega} \mu_n(\{ a \in G_n : y(n) \not \in H(a) \}).
		\]
		But we know that $\mu_i(\{ a \in G_i : y(i) \not \in H(a) \})  \le \frac34$ by (3) of Lemma \ref{lem:graph}. Thus we have (\ref{item:horizontal}).
	\end{proof}


	Therefore, we can draw the diagram as Figure \ref{fig:digram}.
	
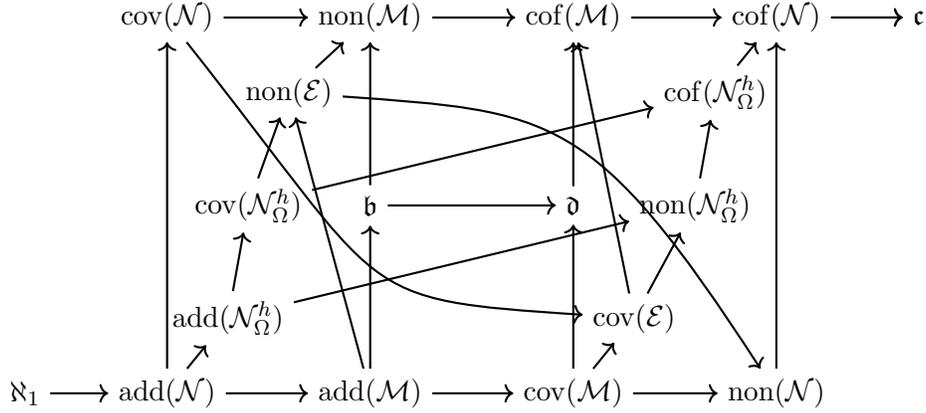
	\begin{figure}[h]
		\[
	\begin{tikzpicture}
		\newcommand{\w}{2.7}
		\newcommand{\h}{2.5}
		
		\node (aleph1) at (-\w*0.7, 0) {$\aleph_1$};
		
		\node (addN) at (0, 0) {$\add(\nul)$};
		\node (covN) at (0, \h*2) {$\cov(\nul)$};
		
		\node (addM) at (\w, 0) {$\add(\meager)$};
		\node (b) at (\w, \h) {$\mathfrak{b}$};
		\node (nonM) at (\w, \h*2) {$\non(\meager)$};
		
		\node (covM) at (\w*2, 0) {$\cov(\meager)$};
		\node (d) at (\w*2, \h) {$\mathfrak{d}$};
		\node (cofM) at (\w*2, \h*2) {$\cof(\meager)$};
		
		\node (nonN) at (\w*3, 0) {$\non(\nul)$};
		\node (cofN) at (\w*3, \h*2) {$\cof(\nul)$};
		
		\node (c) at (\w*3.7, \h*2) {$\frakc$};

		\node (addNh) at (\w*0.3, \h*0.4) {$\add(\scrN^h_\Omega)$};
		\node (cofNh) at (\w*2.7, \h*1.6) {$\cof(\scrN^h_\Omega)$};
		
		\node (covNh) at (\w*0.4, \h*1) {$\cov(\scrN^h_\Omega)$};
		\node (nonNh) at (\w*2.6, \h*1) {$\non(\scrN^h_\Omega)$};
		
		\node (nonE) at (\w*0.6, \h*1.6) {$\non(\mathcal{E})$};
		\node (covE) at (\w*2.3, \h*0.4) {$\cov(\mathcal{E})$};
		
		\draw[thick,->] (aleph1) to (addN);
		\draw[thick,->] (addN) to (covN);
		\draw[thick,->] (addN) to (addM);
		\draw[thick,->] (covN) to (nonM);	
		\draw[thick,->] (addM) to (b);
		\draw[thick,->] (b) to (nonM);
		\draw[thick,->] (addM) to (covM);
		\draw[thick,->] (nonM) to (cofM);
		\draw[thick,->] (covM) to (d);
		\draw[thick,->] (d) to (cofM);
		\draw[thick,->] (b) to (d);
		\draw[thick,->] (covM) to (nonN);
		\draw[thick,->] (cofM) to (cofN);
		\draw[thick,->] (nonN) to (cofN);
		\draw[thick,->] (cofN) to (c);
		\draw[thick,->] (cofN) to (c);

		\draw[thick,->] (nonE) to (nonM);
		\draw[thick,->] (nonE)  ..controls(\w*1.9,\h*1.5) and (\w*2.1, \h*1.5)..  (nonN);
		\draw[thick,->] (covM) to (covE);
		\draw[thick,->] (covN)..controls(\w*1.1,\h*0.5) and (\w*0.9, \h*0.5)..  (covE);
		
		\draw[thick,->] (addN) to (addNh);
		\draw[thick,->] (addNh) to (covNh);
		\draw[thick,->] (covNh) to (cofNh);
		\draw[thick,->] (addNh) to (nonNh);
		\draw[thick,->] (nonNh) to (cofNh);
		\draw[thick,->] (cofNh) to (cofN);
		\draw[thick,->] (covNh) to (nonE);
		\draw[thick,->] (covE) to (nonNh);

		\draw[thick,->] (addM) to (nonE);
		\draw[thick,->] (covE) to (cofM);
	\end{tikzpicture}\]
	\caption{The $\le$ relationships between cardinal invariants}\label{fig:digram}
	\end{figure}
	
	\section{Consistency Results}

	Zapletal \cite{zapletal2008forcing} proved the following theorem.

	\begin{thm}
		Let $P_{\Nho}$ be the idealized forcing of $\Nho$.
		Then, $P_{\Nho}$ is $\omega^\omega$-bounding and it adds no splitting reals.
	\end{thm}

	Therefore, in the model obtained by countable support iteration of $P_{\Nho}$ of length $\omega_2$, we have $\aleph_1 = \frakd < \cov(\Nho)= \frakc = \aleph_2$.
	If we assume large cardinals, then $\frakr=\aleph_1$ holds in the model obtained by the above iteration using Corollary 6.3.8 of \cite{zapletal2008forcing}.
	
	Employing the following theorem due to Klausner and Mejia, we prove the consistency of $\cov(\scrE) < \non(\Nho)$.
	
	\begin{thm}[\cite{klausner2019different}]\label{thm:km}
		Let $c, H, d \in \omega^\omega$ and assume $c > H \ge^* 1$, $d \ge 2$ and $\limsup_{n \to \infty} \frac{\log_{d(n)} H(n)}{d(n)} = \infty$.
		Then there is a proper forcing poset $Q_{c,H}^d$ such that the following conditions hold:
		\begin{enumerate}
			\item $Q_{c,H}^d$ adds a slalom in $\mathcal{S}(c, H)$ that catches infinitely often every real in $\prod c$ in the ground model. In particular, the countable support iteration of $Q_{c,H}^d$ of length $\omega_2$ increases $\frakv^\exists_{c,H}$.
			\item Assume additionally the following conditions:
			\begin{enumerate}
				\item $a, e \in \omega^\omega$, $a > 0$, $e$ goes to infinity.
				\item $\prod_{k < n} a(k) \le d(n)$ and $\prod_{k < n} c^{\triangledown H}(k) \le e(n)$ for all but finitely many $n$.
				Here $c^{\triangledown H}(k)$ means $\abs{[c(k)]^{\le H(k)}}$.
				\item $\lim_{k \to \infty} \min \{ \frac{c^{\triangledown H}(k)}{e(k)}, \frac{a(k)}{d(k)} \} = 0$.
			\end{enumerate}
			Then, $Q_{c,H}^d$ is $(a, e)$-bounding, that is, it forces every $x \in \prod{a}$ there is a slalom in $\mathcal{S}(a, e)$ in the ground model that catches $x$ eventually. Combining this result and Shelah's preservation theorem (Theorem 7.2.19 of \cite{bartoszynski1995set}), we have that the countable support iteration of $Q_{c,H}^d$ of length $\omega_2$ forces $\frakc^\forall_{a,e} = \aleph_1$.
		\end{enumerate}
	\end{thm}

	\begin{lem}\label{lem:covEcallae}
		If $a, e \in \omega^\omega$ with $0 < e \le a$ satisfies $\prod_{n \in \omega} \frac{e(n)}{a(n)} = 0$, then $\cov(\scrE) \le \frakc^\forall_{a,e}$.
	\end{lem}
	\begin{proof}
		Consider the uniform measure $\mu_n$ on $a(n)$ for every $n$ and $\mu$ the product measure on $\prod a$.
		Let $\scrE_{\mu}$ be the $\sigma$-ideal generated by closed sets of $\mu$-measure zero.
		Since $\cov(\scrE) = \cov(\scrE_{\mu})$ holds (see \cite{muz25}), it suffices to show that there is Tukey morphism $(\varphi, \psi) \colon (\prod a, \scrE_\mu, \in) \to (\prod a, \mathcal{S}(a, e), \in^*)$, that is
		\begin{align*}
		&\text{there is $\varphi \colon \prod a \to \prod a$ and $\psi \colon \mathcal{S}(a, e) \to \scrE_\mu$ such that} \\
		&\hspace{1cm}\text{$\varphi(x) \in^* S$ implies $x \in \psi(S)$ for every $x \in \prod a$ and $S \in \mathcal{S}(a, e)$}.
		\end{align*}
		Letting $\varphi(x) = x$ and $\psi(S) = \{ x \in \prod a : (\forall^\infty n)\ x(n) \in S(n) \}$ for $x \in \prod a$ and $S \in \mathcal{S}(a, e)$ suffices.
		Note that $\psi(S)$ is a $F_\sigma$ set.
		Also note that $\mu(\psi(S)) = 0$ for every $S$ by using the assumption of this lemma.
	\end{proof}

	In the following lemma, the sequence $\seq{M_n, N_n : n \in \omega}$ is the one fixed when $\Nho$ is defined.

	\begin{lem}\label{lem:vexistschnonnho}
		Let $c, H \in \omega^\omega$ and $\seq{I_n : n \in \omega}$ be an interval partition.
		Assume $c(n) = \prod_{k \in I_n} N_k$ and $M_{\min{I_{n+1}}} > 2^n H(n) N_0 \dots N_{\min I_n - 1}$ for every $n$.
		Then there is a Tukey morphism from $(\Omega, \Nho, \in)$ to $(\prod c, \mathcal{S}(c, H), \in^\infty)$. In particular $\cov(\Nho) \le \frakc^\exists_{c,H}$ and $\frakv^\exists_{c,H} \le \non(\Nho)$ hold.
	\end{lem}
	\begin{proof}
		We identify the natural number $c(n) = \prod_{k \in I_n} N_k$ and the set $\prod_{k \in I_n} G_k$.
		Write $I_n = [i_n, i_{n+1})$ for $n \in \omega$.
		Define $\varphi \colon \Omega \to \prod c$ and $\psi \colon \mathcal{S}(c, H) \to \Nho$ as follows:
		\begin{align*}
			\varphi(x)(n)  &= x \restrict I_n, \\
			\psi(S) &= \{ x \in \Omega : (\exists^\infty n)\ (x \restrict I_n \in S(n)) \}.
		\end{align*}
		We must show that $\Haus^h_\Omega(\psi(S)) = 0$ for every $S$.
		But by the assumption, we have
		\[
		\frac{H(n) N_0 \dots N_{i_n - 1}}{M_0 \dots M_{i_{n+1}}} \le 2^{-n}.
		\]
		Thus, we have
		\[
		\sum_{n \in \omega} \frac{H(n) N_0 \dots N_{i_n - 1}}{M_0 \dots M_{i_{n+1}}} < \infty.
		\]
		On the other hand, we have $\Haus^h_{\Omega,\infty}(\{ x : x \restrict I_n \in S(n) \}) \le \frac{H(n) N_0 \dots N_{i_n - 1}}{M_0 \dots M_{i_{n+1}}}$ because we have $H(n) N_0 \dots N_{i_n - 1}$ many cylinders. Therefore, we have $\Haus^h_\Omega(\psi(S)) = 0$.
		
		the maps $\phi$ and $\psi$ are Tukey: $\varphi(x) \in^\infty S$ implies $x \in \psi(S)$ due to the definition of $\psi$.
	\end{proof}

	\begin{thm}
		It is consistent that $\cov(\scrE) < \non(\Nho)$.
	\end{thm}
	\begin{proof}
		By induction, we define $c, H, d, a, e \in \omega^\omega$ and an interval partition $\seq{I_n : n \in \omega}$ as follows:
		\begin{enumerate}
			\item Define $d(0) = 2, e(0) = 1$.
			\item Define $H(n) = (d(n))^{n d(n)}$.
			\item Pick $\max I_n$ so that $M_{\min{I_{n+1}}} > 2^n H(n) N_0 \dots N_{\min I_n - 1}$ and $\prod_{k \in I_n} N_k > H(n)$.
			\item Define $c(n) = \prod_{k \in I_n} N_k$.
			\item Define $a(n) = 2 e(n)$.
			\item Define $e(n+1) = 2 \prod_{k \le n} 2^{c(k)}$.
			\item Define $d(n+1) = \prod_{k \le n} a(k)$.
		\end{enumerate}
		Then, we have all assumptions of Theorem \ref{thm:km}, Lemma \ref{lem:covEcallae} and Lemma \ref{lem:vexistschnonnho}.
		Therefore, the countable support iteration of $Q_{c,H}^d$ of length $\omega_2$ forces
		\[
		\aleph_1 = \cov(\scrE)\ = \frakc^\forall_{a,e} < \aleph_2 = \frakv^\exists_{c,H} = \non(\Nho) = \frakc. \qedhere
		\]
	\end{proof}

	Lemma \ref{lem:vexistschnonnho} has another corollary:
	
	\begin{cor}
		In Mathias model, $\cov(\Nho) = \aleph_1$ holds.
		In particular, it is consistent that $\aleph_1 = \cov(\Nho) < \frakb = \fraks = \frakc = \aleph_2$.
	\end{cor}
	\begin{proof}
		By the Laver property, in Mathias model, $\frakc^\forall_{c,H}$ is small (see Section 7.4.A of  \cite{bartoszynski1995set}).
		In particular $\frakc^\exists_{c,H}$ is also small because $\frakc^\exists_{c,H} \le \frakc^\forall_{c,H}$. 
		Thus this corollary follows from Lemma \ref{lem:vexistschnonnho}.
	\end{proof}

		We define a Polish relational system $R$ such that $\frakb(R) = \cov(\Nho)$ and every $\sigma$-centered forcing has goodness for the relation $R$.
	First we define
	$$
	D = \{ f \colon \omega \to \StdSets : \sum_n h(\diam(f(n))) \le 1 \}.
	$$
	Here, $\StdSets$ is the set of all standard sets.
	For $f \in D$ and $x \in \Omega$, we define
	$$
	f \sqsubseteq^{\Nho}_n x \iff \forall k \ge n\ x \not \in f(k).
	$$
	Also, we define $\sqsubseteq^{\Nho} = \bigcup_{n \in \omega} \sqsubseteq^{\Nho}_n$.
	Define $R = (D, \Omega, \sqsubseteq^{\Nho})$.
	Then it can be easily checked that $R$ is a Polish relational system and $\frakb(R) = \cov(\Nho)$.
	As for goodness, we refer the reader \cite{6168e47a-d410-3c15-b326-a2970fc8c433}.
	
	\begin{thm}
		Every $\sigma$-centered forcing has $R$-goodness.
	\end{thm}
	\begin{proof}
		This proof is based on the standard proof that $\sigma$-center forcings do not add the (usual) random real.
		
		Putting
		$$
		A_f = \{ x \in \Omega : \exists^\infty n \ x \in f(n) \},
		$$
		we have the equivalence
		$$
		f \sqsubseteq^{\Nho} x \iff x \not \in A_f.
		$$
		
		Let $P$ be a $\sigma$-centered forcing and $P = \bigcup_{m \in \omega} P_m$, where each $P_m$ is a centered subset of $P$.
		Let $\dot{x}$ be a $P$-name such that $P \forces \dot{x} \in \Omega$.
		Let $N$ be a countable elementary substructure of a sufficient large model containing $\dot{x}$ and $\seq{P_m : m \in \omega}$.
		Also let $f \in D$ be a $\sqsubseteq^{\Nho}$-unbounded element over $N$. We have to show $P \forces \dot{x} \in A_f$.
		Let $p_0 \in P, k_0 \in \omega$ and take $m$ such that $p_0 \in P_m$.
		For each $n \in \omega$, we put
		$$
		S_n = \{ s \in \prod_{i < n} G_i : \forall p \in P_m\ \exists q \le p\ q \forces \dot{x} \restrict n = s \}.
		$$
		We claim here that each $S_n$ is nonempty.
		If $S_n$ were empty, then for every $s \in \prod_{i < n} G_i $, there would exist $p_s \in P_m$ forcing $\dot{x} \restrict n \ne s$.
		Taking a common extension of $\{ p_s : s \in \prod_{i < n} G_i  \}$ would give a contradiction.
		
		Therefore, $\bigcup_n S_n$ is an infinite subtree of a compact tree $\bigcup_n \prod_{i < n} G_i $.
		Thus, by K\"{o}nig's lemma, there is $y \in \Omega$ such that $y \restrict n \in S_n$ for every $n$.
		This $y$ can be taken so that $y \in N$. Therefore, it holds that $y \in A_f$.
		Fix a natural number $k \ge k_0$ such that $y \in f(k)$.
		Since $f(k)$ is open, we can take $\ell$ such that $[y\restrict\ell] \subset f(k)$.
		Since $y \restrict \ell \in S_\ell$, we can take $q \le p$ in $P$ forcing $\dot{x} \restrict \ell = y \restrict \ell$. Therefore $q \forces \dot{x} \in f(k)$.
		Thus, we have $q \forces \dot{x} \in A_f$.
		Since the condition $p$ was arbitrary, we have $P \forces \dot{x} \in A_f$.
	\end{proof}
	
	\begin{cor}
		It is consistent that $\cov(\Nho) < \add(\meager)$.
		Also it is consistent that $\cof(\meager) < \non(\Nho)$.
	\end{cor}
	\begin{proof}
		Consider finite support iterations of Hechler forcing of length $\omega_2$ and length $\omega_2 + \omega_1$, respectively, over the ground model $V \models \CH$.
		Hechler forcing adds a dominating real and a Cohen real.
		
		So in $\omega_2$-iteration, $\add(\meager)$ is $\aleph_2$ and in $\omega_2 + \omega_1$-iteration, $\cof(\meager)$ is $\aleph_1$.
		
		By Theorem 2.7 in \cite{6168e47a-d410-3c15-b326-a2970fc8c433}, in $\omega_2$-iteration, $\cov(\Nho) = \aleph_1$. Also in $\omega_2 + \omega_1$-iteration,  $\non(\Nho) = \aleph_2$.
	\end{proof}

	\section{Questions}

	The following questions remain.

	\begin{question}
		\begin{enumerate}
			\item Do we have $\Con(\add(\nul) < \add(\Nho))$ and $\Con(\cof(\Nho) < \cof(\nul))$?
			\item Do we have $ \add(\Nho) \le \frakb$ and $\frakd \le \cof(\Nho)$?
			\item Do we have $\Con(\add(\Nho) < \cov(\Nho))$ and $\Con(\non(\Nho) < \cof(\Nho))$?	
			\item \label{item:nonNhocovNho} Do we have $\Con(\non(\Nho) < \cov(\Nho))$?
		\end{enumerate}
	\end{question}
	
	\section*{Statements and Declarations}
	
	\subsection*{Competing Interests}
	
	There is no competing interests.
	
	\subsection*{Funding}
	
	Financial support was received from Japan Society for the Promotion of Science.
	
	\printbibliography
\end{document}